\documentclass[leqno,12pt,showkeys]{article}

\usepackage{amsfonts,amsmath,amssymb}
\usepackage{amsthm}
\usepackage{enumerate}
\usepackage{graphics}
\usepackage{graphicx}
\usepackage[all]{xy}
\usepackage{float}%
\usepackage{enumerate}%
\usepackage{amsfonts,amsmath,amssymb}
\usepackage{amsthm}
\usepackage{enumerate}
\usepackage{graphics}
\usepackage{graphicx}
\theoremstyle{definition}
  \newtheorem{theorem}{Theorem}[section] %
  
  \newtheorem{lemma}[theorem]{Lemma}

  \newtheorem{defin}[theorem]{Definition}%
  \newtheorem{problem}[theorem]{Problem}

  \newtheorem{fact}[theorem]{Fact}
  \newtheorem{remark}[theorem]{Remark}%
  
\numberwithin{equation}{section}

\begin{document}

\title
{
Temperedness criterion of the tensor product of parabolic induction for $GL_n$
}
\author{
Yves Benoist, Yui Inoue and Toshiyuki Kobayashi
}
\date{\empty}
\maketitle %
{\abstract{
We give a necessary and sufficient condition for a pair of parabolic subgroups $P$ and $Q$ of $G=GL_n(\mathbb{R})$ such that the tensor product of any two unitarily induced representations from $P$ and $Q$ are tempered. We also give an $L^p$-estimate of matrix coefficients of the regular representations on $L^2(G/L)$ when $L$ is a Levi subgroup of $G$. 
}}

Key words and phrases: tempered representation, reductive group, tensor product, unitary representation, degenerate principal series representation

\section{Statement of main results}
\label{sec:intro}

For two unitary representations $\Pi_j$ on Hilbert spaces $\mathcal{H}_j$ $(j=1,2)$ of a group $G$, the tensor product representation $\Pi_1\otimes\Pi_2$ is a unitary representation of $G$ defined on the Hilbert completion of $\mathcal{H}_1\otimes\mathcal{H}_2$.
Let $\sigma$ and $\tau$ be unitary representations of parabolic subgroups $P$ and $Q$ of $G=GL_n$, respectively, and $\Pi_1=\operatorname{Ind}^G_P(\sigma)$ and $\Pi_2=\operatorname{Ind}^G_Q(\tau)$ the unitary induction, see Section \ref{subsec:RegGH}. 
In this paper we address the following: 

\begin{problem}
\label{q:main}
When is the tensor product representation $\Pi_1\otimes\Pi_2$ tempered?
\end{problem}

Let us explain some background of this problem.

Problem \ref{q:main} asks a coarse information of the spectrum of the tensor product representation $\Pi_1\otimes\Pi_2$. We note that the disintegration of $\Pi_1\otimes\Pi_2$ is far from being understood even for $G=GL_n$ and even when $\sigma$ and $\tau$ are the trivial one-dimensional representations. 
For the very special case where $P$ is a maximal parabolic subgroup and $Q$ is its opposite parabolic subgroup, the tensor product $\operatorname{Ind}^G_P(\mathbf{1})\otimes\operatorname{Ind}^G_Q(\mathbf{1})$ is unitarily equivalent to the regular representation for a reductive symmetric space of $G=GL_n$, for which the Plancherel-type theorem is known up to some complicated vanishing condition of cohomologically induced representations with singular parameters which may affect an answer to Problem \ref{q:main} (\cite[Sect.\ 1]{BK-I}, \cite[Rem.\ 1.4]{BK-IV} and references therein). 
Slightly more generally, when both $P$ and $Q$ are arbitrary maximal parabolic subgroups, Problem \ref{q:main} was solved recently in \cite[Prop.\ 5.9]{BK-II} without the Plancherel-type formula. 
On the other hand, if $P$ or $Q$ is a Borel subgroup, Problem \ref{q:main} has an affirmative answer by the general theory (Remark \ref{rem:tensor}). However, for the general $P$ and $Q$, an answer to Problem \ref{q:main} has not been known. 
In this general setting, we note that the diagonal action of $G$ on $(G\times G)/(P\times Q)$ is not necessarily (real) spherical, and that the multiplicity of irreducible unitary representations in the disintegration of $\operatorname{Ind}^G_P(\mathbf{1})\otimes\operatorname{Ind}^G_Q(\mathbf{1})$ may be infinite, cf.\ \cite{KZ}.

Tempered representations of a locally compact group $G$ are unitary representations that are weakly contained in $L^2(G)$ (Definition \ref{def:temp}). For real reductive Lie groups $G$, irreducible ones
were classified by Knapp and Zuckerman \cite{KZ}, and are cornerstones both in Harish-Chandra's theory of the Plancherel formula of $L^2(G)$ and in Langlands' classification theory of irreducible admissible representations, whereas the Selberg's $1/4$ conjecture for congruence subgroups can be reformulated as the temperedness of certain unitary representations of $SL_2(\mathbb{R})$ and the Gan--Gross--Prasad conjecture is formulated as a branching problem for tempered representations. 
A complete description of pairs $(G,H)$ of real reductive algebraic groups for which $L^2(G/H)$ is not tempered was accomplished in \cite{BK-III}, but such a classification has not been known for non-reductive subgroups $H$ except for a few cases \cite[Cor.\ 5.8]{BK-II}.

In this article, we give a solution to Problem \ref{q:main}. 
We shall prove that the solution depends only on the $G$-conjugacy classes of Levi parts of parabolic subgroups $P$ and $Q$. We introduce the following notation: for a parabolic subgroup $P$ of $GL_n$ with the Levi subgroup $GL_{n_1}\times\cdots\times GL_{n_r}\ (n_1+\cdots+n_r=n)$, we set 
\[
d(P):=\max_{1\leq j\leq r}n_j. 
\]
Then $1\leq d(P)\leq n$ with two extreme cases: $d(P)=1\iff P$ is a Borel subgroup, and $d(P)=n\iff P=G$. We prove:
\begin{theorem}
\label{thm:tensor}
Let $P$ and $Q$ be parabolic subgroups of $G=GL_n(\mathbb{R})$. Then the following three conditions are equivalent:
\begin{itemize}
\item[(i)] The tensor product representation $\operatorname{Ind}^G_P(\mathbf{\sigma})\otimes\operatorname{Ind}^G_Q(\mathbf{\tau})$ is tempered for all unitary representations $\sigma$ of $P$ and $\tau$ of $Q$. 
\item[(ii)] The tensor product representation $\operatorname{Ind}^G_P(\mathbf{1})\otimes\operatorname{Ind}^G_Q(\mathbf{1})$ is tempered. 
\item[(iii)] $d(P)+d(Q)\leq n+1$. 
\end{itemize}
\end{theorem}

An analogous statement holds also for $G=GL_n(\mathbb{C})$.

Theorem \ref{thm:tensor} is derived from the following results about the regular representation on $L^2(G/H)$ where $H$ is not necessarily reductive: 

\begin{theorem}
\label{thm:tempGH}
Let $H$ be a closed subgroup of $G=GL_n(\mathbb{R})$ with finitely many connected components. Assume that the Lie algebra $\mathfrak{h}$ is stable by a split Cartan subalgebra $\mathfrak{a}$ of $\mathfrak{g}$. 
Let $\{e_1,\ldots,e_n\}$ be the standard basis of $\mathfrak{a}^*$ such that $\Delta(\mathfrak{g},\mathfrak{a})=\{\pm(e_i-e_j)\mid1\leq i<j\leq n\}$, and $\{E_1,\ldots,E_n\}$ the dual basis of $\mathfrak{a}$. 
Then the following three conditions are equivalent: 
\begin{itemize}
\item[(i)] $\operatorname{Ind}^G_H(\sigma)$ is tempered for any unitary representation $\sigma$ of $H$. 
\item[(ii)] $L^2(G/H)$ is tempered. 
\item[(iii)] $\dim\operatorname{Image}(\operatorname{ad}(E_i)\colon\mathfrak{h}\to\mathfrak{h})\leq n-1$ for all $i\ (1\leq i\leq n)$. 
\end{itemize} 
\end{theorem}

In general $0\leq\dim\operatorname{Image}(\operatorname{ad}(E_i)\colon\mathfrak{h}\to\mathfrak{h})\leq 2n-1$ for any $\mathfrak{a}$-stable Lie algebra $\mathfrak{h}$ and any $i\ (1\leq i\leq n)$. 
Theorem \ref{thm:tempGH} justifies the ``heuristic recipe'' in \cite[Rem.\ 5.7]{BK-II} for subgroups of three-by-three block matrix form. 


Our proof relies on the temperedness criterion (Fact \ref{fact:BK}), which was established in \cite{BK-I,BK-II} by an analytic and dynamical approach in the general case. The criterion can be reduced to finitely many inequalities 
arising from the edges of convex polyhedral cones, 
actually $2^n$ inequalities in our setting. 
To solve Problem \ref{q:main}, we still need to analyze the $2^n$ inequalities. A number of combinatorial techniques were proposed in \cite{BK-II,BK-III}, and among them, Theorem \ref{thm:tempGH} was proved in the special setting where $\mathfrak{h}$ is a subalgebra of three-by-three block matrix form (\cite[Cor.\ 5.8]{BK-II}). 
The new technical ingredients in this article include yet another combinatorial trick which reduces $2^n$ inequalities to very simple $n$ inequalities (the third condition in Theorem \ref{thm:tempGH}). 
The same technique also yields an $L^p$-estimate of the matrix coefficients of the regular representation $L^2(G/H)$ when $H$ is reductive, see Theorem \ref{prop:Lp}. 

This article is organized as follows.
In Section \ref{sec:pre}, we review the Herz majoration principle and the temperedness criterion in a general setting. Section \ref{sec:pfmain} provides 
a proof of Theorems \ref{thm:tensor} and \ref{thm:tempGH}, postponing a combinatorial proof of Lemma \ref{lem:iImax} until Section \ref{sec:comb}. In Section \ref{sec:Append}, we discuss Problem \ref{q:main} for any simple groups under the assumption that $Q$ is the opposite parabolic subgroup of $P$.

\section{Preliminaries}
\label{sec:pre}

In this section we fix some notations and recall the previous results on unitary representations that will be needed later.

\subsection{Regular representations}
\label{subsec:RegGH}

For an $m$-dimensional manifold $X$, we denote by $\mathcal{L}_{vol}\equiv\mathcal{L}_{X,vol}:=\left|\bigwedge^m(T^*X)\right|$ the density bundle of X, and by $L^2(X)$ the Hilbert space of square integrable sections for the half-density bundle $\mathcal{L}_{vol}^{1/2}$.
Suppose a Lie group $G$ acts continuously on $X$. Then $G$ acts equivariantly on the half-density bundle $\mathcal{L}_{vol}^{1/2}$, and one has naturally a unitary representation $\lambda_X$ of $G$ on $L^2(X)$, referred to as the {\it regular representation}. Associated to a unitary representation $(\sigma,W)$ of a closed subgroup $H$ of $G$, the unitary induction $\operatorname{Ind}^G_H(\sigma)$ is defined as a unitary representation of $G$ on the Hilbert space of square integrable sections for the $G$-equivariant Hilbert bundle $(G\times_HW)\otimes\mathcal{L}_{vol}^{1/2}$ over the homogeneous space $G/H$. By definition, $\operatorname{Ind}^G_H(\mathbf{1})$ is the regular representation $\lambda_{G/H}$ on $L^2(G/H)$, where $\mathbf{1}$ denotes the trivial one-dimensional representation of $H$.

\subsection{Tempered representations}
\label{subsec:temp}

Let $(\pi,\mathcal{H})$ and $(\pi',\mathcal{H}')$ be unitary representations of a locally compact group $G$. We say $\pi$ is {\it weakly contained} in $\pi'$, to be denoted by $\pi\prec\pi'$ if for every $v\in\mathcal{H}$ the matrix coefficient $(\pi(g)v,v)$ can be approximated uniformly on compact subsets of $G$ by a sequence of finite sums of functions $(\pi'(g)u_j,u_j)$ with $u_1,\cdots,u_k\in\mathcal{H}'$. 

\begin{defin}
\label{def:temp}
A unitary representation $\pi$ of $G$ is called {\it tempered} if $\pi$ is weakly contained in the (left) regular representation $\lambda_G$ on $L^2(G)$. 
\end{defin}

When $G$ is a semisimple Lie group, $\pi$ is tempered if and only if $\pi$ is almost $L^2$, see \cite{CHH}. Here we recall:
\begin{defin}
\label{def:almostLp}
Let $p\geq1$. A unitary representation $(\pi,\mathcal{H})$ of $G$ is said to be {\it almost} $L^p$ if there exists a dense subset $D$ of $\mathcal{H}$ for which the matrix coefficients $g\mapsto(\pi(g)u,v)$ are in $L^{p+\varepsilon}(G)$ for all $\varepsilon>0$ and all $u,v\in D$. 
\end{defin}

\begin{remark}
\label{rem:tensor}
Temperedness is closed under induction and restrictions of unitary representations.  
Moreover, if $\pi$ is tempered, then the tensor product representation $\pi\otimes\sigma$ is tempered for any unitary representation $\sigma$ of $G$. In fact, if $\pi\prec\lambda_G$, then $\pi\otimes\sigma\prec\lambda_G\otimes\sigma$. Since $\lambda_G\otimes\sigma$ is a multiple of $\lambda_G$ (\cite[Cor.\ E.2.6]{BHV}), one concludes $\pi\otimes\sigma\prec\lambda_G$.
\end{remark}

We recall a classical lemma called ``Herz majoration principle'', see \cite[Sect.\ 6]{BG}:

\begin{lemma}[{\cite[Lem.\ 3.2]{BK-II}}]
\label{lem:Herz}
Let $G$ be a semisimple Lie group with finitely many connected components such that the identity component has finite center, and $H$ a closed subgroup of $G$. If the regular representation $\lambda_{G/H}$ is tempered, then the induced representation $\operatorname{Ind}^G_H(\sigma)$ is tempered for any unitary representation $\sigma$ of $H$. 
\end{lemma}

\subsection{Temperedness criterion for $L^2(G/H)$}
\label{subsec:BK}
A Lie algebra is said to be {\it algebraic} if it is isomorphic to the Lie algebra of an affine algebraic group, or equivalently, the image of the adjoint representation $\operatorname{ad}\colon\mathfrak{h}\to\operatorname{End}(\mathfrak{h})$ is the Lie algebra of an algebraic subgroup of $\operatorname{Aut}(\mathfrak{h})$, see \cite{Gt}.
A subalgebra $\mathfrak{a}$ is said to be {\it split} if $\operatorname{ad}(H)\in\operatorname{End}(\mathfrak{h})$ is diagonalizable over $\mathbb{R}$ for every $H\in\mathfrak{a}$.
Let $\mathfrak{a}$ be a maximal split abelian  subalgebra in an (algebraic) Lie algebra $\mathfrak{h}$. 
Such $\mathfrak{a}$ is unique up to conjugation, and we denote by $\operatorname{rank}_\mathbb{R}\mathfrak{h}$ its dimension when $\mathfrak{h}$ is a semisimple Lie algebra. 

Let $V$ be a finite-dimensional representation of $\mathfrak{h}$. 
Following \cite{BK-I,BK-II}, we define a non-negative function $\rho_V$ on $\mathfrak{a}$ by 
\[
\rho_V(Y):=\frac{1}{2}\sum_{\lambda\in\Delta(V,\mathfrak{a})}m_\lambda|\lambda(Y)|\quad\text{for }Y\in\mathfrak{a},
\]
where $\Delta(V,\mathfrak{a})$ is the set of weights of $\mathfrak{a}$ in $V$ and $m_\lambda$ denotes the dimension of the corresponding weight space $V_\lambda$. The function $\rho_V$ is continuous and is piecewise linear i.e.\ there exist finitely many convex polyhedral cones which covers $\mathfrak{a}$ and on which $\rho_V$ is linear, see \cite[Sect.\ 4.7]{BK-I}. 
We set
\begin{align}
\label{eqn:BKconstant}
p_V:=\max_{Y\in\mathfrak{a}\setminus\{0\}}\frac{\rho_\mathfrak{h}(Y)}{\rho_V(Y)}. 
\end{align}

\begin{fact}
\label{fact:BK}
Let $G$ be a linear semisimple Lie group, and $H$ an algebraic subgroup.
\begin{itemize}
\item[(1)](\cite[Thm.\ 2.9]{BK-II})
One has the equivalence:
\[
L^2(G/H)\text{ is tempered }\iff2\rho_\mathfrak{h}\leq\rho_\mathfrak{g}\text{ on }\mathfrak{a}. 
\]
\item[(2)](\cite[Thm.\ 4.1]{BK-I})
Let $p$ be a positive even integer. If $H$ is reductive, one has the equivalence:
\[
L^2(G/H)\text{ is almost }L^p\iff p_{\mathfrak{g}/\mathfrak{h}}\leq p-1.
\]
\end{itemize}
\end{fact}

The inequality in Fact \ref{fact:BK} can be checked only at finitely many points in $\mathfrak{a}$, 
namely, at the generators of the edges of the convex polyhedral cones, 
as we shall see in Lemma \ref{lem:edge} below in the setting we need. 

\section{Proof of Theorems \ref{thm:tensor} and \ref{thm:tempGH}}
\label{sec:pfmain}

In this section, we show the main results by using the temperedness criterion (Fact \ref{fact:BK}) and some combinatorial lemmas. We postpone the proof of Lemma \ref{lem:iImax} until Section \ref{sec:comb}. 

Suppose $\mathfrak{g}=\mathfrak{gl}_n(\mathbb{R})$ and $\mathfrak{h}$ is an $\mathfrak{a}$-invariant subalgebra as in the setting of Theorem \ref{thm:tempGH}. 
Since split Cartan subalgebras $\mathfrak{a}$ are conjugate to each other by inner automorphisms, we may and do assume $\mathfrak{a}=\bigoplus_{i=1}^n\mathbb{R} E_{ii}$, where $E_{ij}$ denotes the matrix unit.

For $1\leq i,j\leq n$, we set
\begin{align}
\label{eqn:epsilonij}
\varepsilon_{ij}\equiv\varepsilon_{ij}(\mathfrak{h}):=\dim_\mathbb{R}(\mathfrak{h}\cap\mathbb{R}E_{ij})\in\{0,1\}. 
\end{align}

By the weight decomposition of $\mathfrak{h}$ with respect to $\mathfrak{a}$, one sees
\begin{align}
\label{eqn:rhi}
\dim\operatorname{Image}(\operatorname{ad}(E_{ii})\colon\mathfrak{h}\to\mathfrak{h})=\sum_{j\in\{1,\ldots,n\}\setminus\{i\}}(\varepsilon_{ij}+\varepsilon_{ji})=2\rho_\mathfrak{h}(E_{ii}). 
\end{align}

Since $\rho_\mathfrak{g}(E_{ii})=n-1$, the condition (iii) in Theorem \ref{thm:tempGH} amounts to
\[
2\rho_\mathfrak{h}(E_{ii})\leq \rho_\mathfrak{g}(E_{ii})\quad\text{for all }i\ (1\leq i\leq n). 
\]

\subsection{Reduction to finite inequalities}
\label{subsec:comb}

The temperedness criterion (Fact \ref{fact:BK}) is given by the inequality on $\mathfrak{a}$, which reduces to a finite number of inequalities on the generators of convex polyhedral cones. This is Lemma \ref{lem:edge} below which reduces to $2^n$ inequalities. A further combinatorial argument reduces to $n$ inequalities (Lemma \ref{lem:iImax}).

For a non-empty subset $I\subset\{1,\ldots,n\}$, we set $E_I:=\sum_{i\in I}E_{ii}$. 
Then $E_I=E_{ii}$ if $I=\{i\}$; $E_I$ generates the center $\mathfrak{z}(\mathfrak{g})$ of $\mathfrak{g}$ if $I=\{1,2,\ldots,n\}$.

\begin{lemma}
\label{lem:edge}
The condition (ii) in Theorem \ref{thm:tempGH} is equivalent to
\begin{align}
\label{eqn:EIrho}
2\rho_\mathfrak{h}(E_I)\leq\rho_\mathfrak{g}(E_I)\quad\text{for all } I\subset\{1,\ldots,n\}. 
\end{align}
\end{lemma}

\begin{proof}
By the temperedness criterion (Fact \ref{fact:BK}), the condition (ii) in Theorem \ref{thm:tempGH} is given by $2\rho_\mathfrak{h}\leq\rho_\mathfrak{g}$ on $\mathfrak{a}/\mathfrak{z}(\mathfrak{g})$. 
Thus it suffices to show
\begin{align}
\label{eqn:edge}
\max_{0\neq Y\in\mathfrak{a}/\mathfrak{z}(\mathfrak{g})}\frac{\rho_\mathfrak{h}(Y)}{\rho_\mathfrak{g}(Y)}=\max_{I\subsetneq\{1,\ldots,n\}}\frac{\rho_\mathfrak{h}(E_I)}{\rho_\mathfrak{g}(E_I)}.
\end{align}
To see the non-trivial inequality $\leq$,
we begin with the dominant chamber $\overline{\mathfrak{a}_+}=\{\operatorname{diag}(x_1,\ldots,x_n)\colon x_1\geq\cdots\geq x_n\}$. Since both $\rho_\mathfrak{h}$ and $\rho_\mathfrak{g}$ are linear on $\overline{\mathfrak{a}_+}$, the restriction of the function $\rho_\mathfrak{h}/\rho_\mathfrak{g}$ to the line segment $tY+(1-t)Z\ (Y,Z\in\overline{\mathfrak{a}_+}\setminus\mathfrak{z}(\mathfrak{g}))$ is a linear fractional function of $t\ (0\leq t\leq 1)$, which attains its maximum either at $t=0$ or $t=1$. 
An iteration of the argument tells that the maximum of $\rho_\mathfrak{h}/\rho_{\mathfrak{g}}$ on $(\overline{\mathfrak{a}_+}/\mathfrak{z}(\mathfrak{g}))\setminus\{0\}$ is attained at one of the edges of the convex polyhedral cone $\overline{\mathfrak{a}_+}/\mathfrak{z}(\mathfrak{g})$, namely, at $\mathbb{R}_+E_I$ for some $I=\{1,2,\ldots,k\}$ with $1\leq k\leq n-1$. 

Similar argument applies to the other Weyl chambers. 
\end{proof}

The following lemma tells that it suffices to use $E_I$ with $\# I=1$ for ``witness vectors'' (\cite{BK-III}) in our setting, 
and will be proved in Section \ref{sec:comb}. 
\begin{lemma}
\label{lem:iImax}
If $2\rho_\mathfrak{h}(E_{ii})\leq\rho_\mathfrak{g}(E_{ii})$ for all $i$ ($1\leq i\leq n$), then (\ref{eqn:EIrho}) holds. 
\end{lemma}

\subsection{Proof of Theorem \ref{thm:tempGH}}
\label{subsec:pftempGH}

The equivalence (i) $\iff$ (ii) in Theorem \ref{thm:tempGH} follows from the Herz majoration principle (Lemma \ref{lem:Herz}). Let us verify the equivalence (ii) $\iff$ (iii). 
We may and do assume that $\mathfrak{h}$ contains $\mathfrak{a}=\sum_{i=1}^n\mathbb{R}E_{ii}$. In fact, if $\mathfrak{h}$ is stable by $\mathfrak{a}$, then $\tilde{\mathfrak{h}}:=\mathfrak{h}+\mathfrak{a}$ is a Lie subalgebra containing $\mathfrak{a}$. We write $\tilde{H}$ for the connected subgroup of $G$ with Lie algebra $\tilde{\mathfrak{h}}$. Then $L^2(G/H)$ is tempered if and only if $L^2(G/\tilde{H})$ is tempered by \cite[Cor.\ 3.3]{BK-II}. Moreover, $\operatorname{Image}(\operatorname{ad}(E_{ii})\colon\mathfrak{h}\to\mathfrak{h})$ remains the same if we replace $\mathfrak{h}$ with $\tilde{\mathfrak{h}}$, hence the conditions (ii) and (iii) in Theorem \ref{thm:tempGH} are unchanged. 
Now one has the equivalences: 
\begin{alignat*}{4}
\text{(ii)}\iff&2\rho_\mathfrak{h}(Y)&&\leq\rho_\mathfrak{g}(Y)&\quad&(^\forall Y\in\mathfrak{a}) &\quad&\text{by Fact \ref{fact:BK}}\\
\iff&2\rho_\mathfrak{h}(E_I)&&\leq\rho_\mathfrak{g}(E_I)&\quad&(^\forall I\subset\{1,\ldots,n\})&\quad&\text{by Lemma \ref{lem:edge}}\\
\iff&2\rho_\mathfrak{h}(E_{ii})&&\leq\rho_\mathfrak{g}(E_{ii})&\quad&(1\leq{}^\forall i\leq n)&\quad&\text{by Lemma \ref{lem:iImax}},
\end{alignat*}
which is equivalent to (iii). Thus Theorem \ref{thm:tempGH} is proved.

\subsection{Proof of Theorem \ref{thm:tensor}}
\label{subsec:pftensor}

Without loss of generality, we may and do assume that $P$ and $Q$ are standard parabolic subgroups with Levi subgroups $GL_{n_1}\times\cdots\times GL_{n_r}$ and $GL_{m_1}\times\cdots\times GL_{m_s}$, respectively. 
Let $w:=\sum_{i=1}^nE_{i\ n+1-i}\in G$, a representative of the longest element of the Weyl group $W(\mathfrak{g},\mathfrak{a})$. 
Then $Q^o:=w^{-1}Qw$ is a parabolic subgroup of $G$ with Levi subgroup $GL_{m_s}\times\cdots\times GL_{m_1}$, and $PQ^o$ is open dense in $G$, hence the diagonal map $G\to G\times G,\ g\mapsto (g,g)$ induces an open dense embedding $\iota\colon G/H\hookrightarrow G/P\times G/Q^o$, where $H:=P\cap Q^o$. Thus the tensor product representation $\operatorname{Ind}^G_P(\mathbf{1})\otimes\operatorname{Ind}^G_Q(\mathbf{1})\simeq\operatorname{Ind}^G_P(\mathbf{1})\otimes\operatorname{Ind}^G_{Q^o}(\mathbf{1})$ is unitarily equivalent to $L^2(G/H)$
via the $G$-isomorphism of the equivariant line bundles $\iota^*(\mathcal{L}_{G/P,vol}\otimes\mathcal{L}_{G/Q^o,vol})\simeq\mathcal{L}_{G/H,vol}$.

We define integers $N(a)\ (0\leq a\leq r)$ and $M(b)\ (0\leq b\leq s)$ by
\[
N(a):=\sum_{j=1}^an_j\quad(1\leq a\leq r),\ M(b):=\sum_{j=1}^bm_{s+1-j}\quad(1\leq b\leq s),
\]
and set $N(0)=M(0)=0$. 
We note $N(r)=M(s)=n$. By definition, for each $1\leq i\leq n$, there exist uniquely $a(i)\in\{1,\ldots,r\}$ and $b(i)\in\{1,\ldots,s\}$ such that \
\begin{align}
\label{eqn:NMi}
N(a(i)-1)<i\leq N(a(i))\text{ and }M(b(i)-1)<i\leq M(b(i)).
\end{align} 
By definition, one has for $1\leq i,j\leq n$,
\begin{align*}
E_{ij}\in\mathfrak{p}&\iff N(a(i)-1)<j,&\ E_{ij}\in\mathfrak{q}^o&\iff j\leq M(b(i)),\\
E_{ji}\in\mathfrak{p}&\iff j\leq N(a(i)),&\ E_{ji}\in\mathfrak{q}^o&\iff M(b(i)-1)<j.
\end{align*}
Since the Lie algebra $\mathfrak{h}$ of $H$ is equal to $\mathfrak{p}\cap\mathfrak{q}^o$, (\ref{eqn:rhi}) shows
\begin{align*}
2\rho_\mathfrak{h}(E_{ii})=&(M(b(i))-N(a(i)-1)-1)+(N(a(i))-M(b(i)-1)-1)\\
=&n_{a(i)}+m_{s-b(i)+1}-2.
\end{align*}
Since $\mathfrak{h}$ contains $\mathfrak{a}$, we can apply Theorem \ref{thm:tempGH}, and conclude that $L^2(G/H)$ is tempered if and only if 
\begin{align}
\label{eqn:nmmax}
n_{a(i)}+m_{s-b(i)+1}\leq n+1\quad\text{for all }i\ (1\leq i\leq n). 
\end{align}
We claim (\ref{eqn:nmmax}) holds if and only if 
\begin{align}
\label{eqn:dPQ}
d(P)+d(Q)\leq n+1. 
\end{align}
The implication (\ref{eqn:dPQ}) $\Rightarrow$ (\ref{eqn:nmmax}) is obvious. 
To see the converse implication, we take $a\in\{1,\ldots,r\}$ and $b\in\{1,\ldots,s\}$ such that $n_a=d(P)$ and $m_{s+1-b}=d(Q)$. Then the subsets $\{N(a-1)+1,\ldots,N(a)\}$ and $\{M(b-1)+1,\ldots,M(b)\}$ of $\{1,2,\ldots,n\}$ have $d(P)$ and $d(Q)$ elements, respectively. If (\ref{eqn:dPQ}) fails, then one finds a common element, say $i$. By (\ref{eqn:NMi}), $a=a(i)$ and $b=b(i)$, hence (\ref{eqn:nmmax}) fails. Thus Theorem \ref{thm:tensor} is proved.

\section{Proof of Lemma \ref{lem:iImax}}
\label{sec:comb}

In this section, we show Lemma \ref{lem:iImax}, hence complete the proof of Theorems \ref{thm:tensor} and \ref{thm:tempGH}. 
Actually, we prove a generalization of Lemma \ref{lem:iImax} (see Lemma \ref{lem:Bp} below) which will be used also in an $L^p$ estimate of matrix coefficients (Theorem \ref{prop:Lp}).

\subsection{Reduction to quadratic inequalities}
We recall that $\mathfrak{h}$ is a Lie subalgebra of $\mathfrak{g}=\mathfrak{gl}_n(\mathbb{R})$ containing the Lie algebra $\mathfrak{a}$ of diagonal matrices. We also recall the notation $E_I=\sum_{i\in I}E_{ii}\in\mathfrak{a}$ for a subset $I$ of $\{1,\ldots,n\}$. We prove the following. 

\begin{lemma}\label{lem:Bp}
Suppose $p$ is an even integer $\geq2$. Then the inequality
\[
p\rho_\mathfrak{h}(E_I)\leq (p-1)\rho_\mathfrak{g}(E_I)
\]
is true for all subsets $I$ as soon as it is true when $I$ is a singleton.
\end{lemma}

\begin{remark}
An analogous statement to Lemma \ref{lem:Bp} fails for $p=3$, for instance, when $n=4$ and $\mathfrak{h}$ is a maximal parabolic subalgebra of dimension $12$. 
\end{remark}

Let $\{f_1,\ldots,f_n\}$ be the standard basis of $\mathbb{R}^n$, and  $W_j=\mathbb{R}f_j\ (1\leq j\leq n)$.
By definition, $\mathfrak{a}$ is a subalgebra of $\mathfrak{h}$ which is of the form $\mathfrak{gl}(W_1)\oplus\cdots\oplus\mathfrak{gl}(W_n)$. Let $\mathfrak{l}$ be a maximal reductive subalgebra of $\mathfrak{h}$ of this type, namely, maximal among all reductive subalgebras of $\mathfrak{h}$ containing $\mathfrak{a}$ which is of the form $\mathfrak{gl}(V_1)\oplus\cdots\oplus\mathfrak{gl}(V_r)$ for some direct sum decomposition $\mathbb{R}^n=V_1\oplus\cdots\oplus V_r$ where each $V_j$ is spanned by a subset of the standard basis. We set
\[
n_k:=\dim V_k \text{ and } m_k\equiv m_k(I):=\#\{i\in I\mid f_i\in V_k\} \text{ so that}
\]
\[
n_1+\cdots+n_r=n,\ m_1+\cdots+m_r=\#I \text{ and } 0\leq m_k\leq n_k,\ \text{for all }k\leq r.
\]
Similarly to (\ref{eqn:epsilonij}), we set $\varepsilon_{k\ell}:=1$ if $\operatorname{Hom}_\mathbb{R}(V_\ell,V_k)\subset \mathfrak{h}$, and $\varepsilon_{k\ell}:=0$ otherwise. 
One has $\varepsilon_{kk}=1\ (1\leq k\leq r)$ and $\varepsilon_{k\ell}+\varepsilon_{\ell k}\in\{0,1\}$ by the maximality of $\mathfrak{l}$. 
To compute $\rho_\mathfrak{g}(E_I)$ and $\rho_\mathfrak{h}(E_I)$, we first observe 
\[
\operatorname{ad}(E_{aa})E_{ij}=(\delta_{ai}-\delta_{aj})E_{ij},
\]
 where $\delta_{ab}$ denotes the Kronecker delta. Hence one has
\[
\operatorname{ad}(E_I)E_{ij}=
\begin{cases}
E_{ij}&\text{if }i\in I, j\not\in I,\\
-E_{ij}&\text{if }i\not\in I, j\in I,\\
0&\text{otherwise.}
\end{cases}
\]
Summing up the absolute values of the eigenvalues of $\operatorname{ad}(E_I)$ on $\mathfrak{g}$, one has
\[
\rho_\mathfrak{g}(E_I)=\#I(n-\#I)=\sum_{1\leq k,\ell\leq r}m_k(n_\ell-m_\ell).
\]
Similarly, summing up the absolute values of the eigenvalues of $\operatorname{ad}(E_I)$ on $\mathfrak{h}$, one has from the definition of $\varepsilon_{k\ell}$ the following formula:
\[
2\rho_\mathfrak{h}(E_I)=\sum_{1\leq k, \ell\leq r}\varepsilon_{k\ell}(m_k(n_\ell-m_\ell)+m_\ell(n_k-m_k))
=\sum_{1\leq k,\ell\leq r}b_{k\ell}m_k(n_\ell-m_\ell),
\]
where $b_{kk}=2$ and $b_{k\ell}=\varepsilon_{k\ell}+\varepsilon_{\ell k}\ (k\neq\ell)$. Hence, setting $a_{kk}=1$ and $a_{k\ell}=1+\frac{p}{2}(\varepsilon_{k\ell}+\varepsilon_{\ell k}-2)$, one has
\[
p\rho_\mathfrak{h}(E_I)-(p-1)\rho_\mathfrak{g}(E_I)=\sum_{1\leq k,\ell\leq r}a_{k\ell}m_k(n_\ell-m_\ell). 
\]
Since $\varepsilon_{k\ell}+\varepsilon_{\ell k}\in\{0,1\}$, we see $a_{k\ell}\in\{1-p,1-\frac{p}{2}\}$ for all $k\neq \ell$, in particular, $a_{k\ell}$ are non-positive integers when $p$ is even. 
Hence Lemma \ref{lem:Bp} follows from Lemma \ref{lem:quadraineq} below.

\subsection{Quadratic inequalities}
This section is independent of the previous one. We forget about Lie algebras. We fix integers $r\geq 1,\ n_1,\ldots,n_r\geq1$ and $(a_{k\ell})_{1\leq k,\ell\leq r}$ a symmetric matrix with integer coefficients which are equal to $1$ on the diagonal and are non-positive outside the diagonal:
\[
a_{k\ell}=a_{\ell k}\in-\mathbb{N} \text{ for all } k\neq \ell \text{ and } a_{\ell\ell}=1 \text{ for all }\ell. 
\]
Here, we used the notation $\mathbb{N}=\{0,1,2,\ldots\}$. We denote by $\mathbf{e}_\ell\in\mathbb{N}^r$ the $r$-tuple $\mathbf{e}_\ell=(\delta_{k,\ell})_{1\leq k\leq r}$. We fix $\mathbf{n}=(n_1,\ldots,n_r)\in\mathbb{N}^r$, and set
\[
f(\mathbf{m})=\sum_{1\leq k,\ell\leq r}a_{k\ell}m_k(n_\ell-m_\ell) \text{ for } \mathbf{m}=(m_1,\ldots,m_r)\in\mathbb{N}^r. 
\]
For instance, one has $f(\mathbf{e}_{\ell_0})=n_{\ell_0}-1+\sum_{\ell\neq\ell_0}a_{\ell_0\ell}n_\ell$.

\begin{lemma}\label{lem:quadraineq}
Assume that $f(\mathbf{e}_\ell)\leq 0$ for all $1\leq \ell \leq r$. Then one has $f(\mathbf{m})\leq0$ for all $\mathbf{m}$ in $\mathbb{N}^r$ with $\mathbf{n}-\mathbf{m}\in\mathbb{N}^r$. 
\end{lemma}

\begin{proof}
We argue by induction on $s:=\sum_km_k$. Our assumption tells us that the conclusion is true for $s\leq 1$. We assume $s\geq2$ and distinguish two cases.
\begin{itemize}
\setlength{\itemindent}{0.8cm}
\item[{\bf Case 1}]: there exists $1\leq \ell\leq r$ such that $\sum_k a_{k\ell}m_k\geq1$. 

In this case, we fix such an $\ell$. Since $a_{k\ell}\leq 0$ for all $k\neq\ell$ and $a_{\ell\ell}>0$, we can write $\mathbf{m}=\mathbf{m'}+\mathbf{e}_\ell$ with $\mathbf{m}'\in\mathbb{N}^r$. Since $a_{\ell\ell}=1$ and $a_{\ell k}=a_{k\ell}$, one has
\[
f(\mathbf{m})=f(\mathbf{m}')+f(\mathbf{e}_\ell)+2-2\sum_k a_{k\ell}m_k. 
\]
Using our assumptions and the induction hypothesis, we get $f(\mathbf{m})\leq0$. 

\item[{\bf Case2}]: For all $1\leq \ell\leq r$, one has $\sum_k a_{k\ell}m_k\leq0$. 

In this case, since $n_\ell-m_\ell\geq0$ for all $\ell$, the inequality $f(\mathbf{m})\leq 0$ follows directly from the definition of $f(\mathbf{m})$.
\end{itemize}

Since the coefficients $a_{k\ell}$ are integers, these two cases are the only possibilities and this ends the proof of Lemma \ref{lem:quadraineq} and hence of Lemma \ref{lem:Bp}.
\end{proof}

\subsection{$L^p$-estimate of matrix coefficients}
\label{subsec:Lp}

When $H$ is reductive, Lemma \ref{lem:Bp} determines an explicit bound of $p$ such that $L^2(G/H)$ is almost $L^p$. We end this section with the following: 

\begin{theorem}
\label{prop:Lp}
Let $n_1+\cdots+n_r\leq n$ and $p\in2\mathbb{N}$. We set $m:=\max(n_1,\ldots,n_r)$.
Then one has the equivalence:
\begin{itemize}
\item[(i)] $L^2(GL_n(\mathbb{R})/(GL_{n_1}(\mathbb{R})\times\cdots\times GL_{n_r}(\mathbb{R})))$ is almost $L^p$. 
\item[(ii)] $m\leq n-\frac{n-1}{p}$. 
\end{itemize}
\end{theorem}

The case $p=2$ was proved in \cite[Thms.\ 1.4 and 3.1]{BK-III}. 

\begin{proof}
For $\mathfrak{h}=\mathfrak{gl}_{n_1}(\mathbb{R})\oplus\cdots\oplus\mathfrak{gl}_{n_r}(\mathbb{R})$, 
we set
\begin{align}
\label{eqn:c0}
c\equiv c(\mathfrak{h}):=\min_{1\leq i\leq n}\frac{\rho_\mathfrak{g}(E_{ii})}{\rho_\mathfrak{h}(E_{ii})}=\frac{2(n-1)}{\max_{1\leq i\leq n}\dim\operatorname{Image}(\operatorname{ad}(E_{ii})\colon\mathfrak{h}\to\mathfrak{h})}.
\end{align}

By definition, $c(\mathfrak{h})=\frac{n-1}{m-1}$, and therefore $p_{\mathfrak{g}/\mathfrak{h}}=\frac{1}{c(\mathfrak{h})-1}=\frac{m-1}{n-m}$ by (\ref{eqn:edge}) and Lemma \ref{lem:Bp}. 
Then Theorem \ref{prop:Lp} follows from the criterion given in Fact \ref{fact:BK} (2). 
\end{proof}

\section{Appendix --- the opposite parabolic case}
\label{sec:Append}

So far we have discussed the temperedness of the tensor product representations $\Pi_1\otimes\Pi_2$ when $\Pi_1$ and $\Pi_2$ are induced from unitary representations of parabolic subgroups $P$ and $Q$ of $G=GL_n$, respectively, see Problem \ref{q:main}. In this appendix, we discuss Problem \ref{q:main} for other reductive groups $G$ under the assumption that $Q$ is the opposite parabolic subgroup of $P$. 
In this case $P\cap Q$ is a reductive subgroup, and we can utilize the list of pairs $(G,H)$ of real reductive algebraic groups for which $L^2(G/H)$ is non-tempered \cite{BK-III}. The main result of this section is the classification of the pairs $(G, P)$ for which $\operatorname{Ind}^G_P({\bf 1})\otimes\operatorname{Ind}^G_Q({\bf 1})$ is tempered, see Theorem \ref{thm:append}. 

To describe the classification, we fix some notation. For a reductive Lie group $L$, there is a unique maximal connected normal non-compact semisimple subgroup, denoted by $L_{ns}$, to which we refer as the {\it non-compact semisimple part} of $L$. Its Lie algebra $\mathfrak{l}_{ns}$ is an ideal of $\mathfrak{l}$ contained in $[\mathfrak{l},\mathfrak{l}]$.

In what follows, we assume that the real simple Lie group $G$ has at most finitely many connected components and that the identity component has finite center.

\begin{theorem}\label{thm:append}
Let $G$ be a non-compact real simple Lie group, $P$ a proper parabolic subgroup, and $Q$ the opposite parabolic. 
We set $L:=P\cap Q$, which is a Levi subgroup of $P$ (and also of $Q$). We write $\mathfrak{l}$ for the Lie algebra of $L$.
Then the following three conditions on the pair $(G,P)$ are equivalent: 
\begin{itemize}
\item[(i)] The tensor product representation $\operatorname{Ind}^G_P(\mathbf{\sigma})\otimes\operatorname{Ind}^G_Q(\mathbf{\tau})$ is tempered for all unitary representations $\sigma$ of $P$ and $\tau$ of $Q$. 
\item[(ii)] The tensor product representation $\operatorname{Ind}^G_P(\mathbf{1})\otimes\operatorname{Ind}^G_Q(\mathbf{1})$ is tempered. 
\item[(iii)]
One of the following conditions holds: 
\begin{enumerate}
\setlength{\itemindent}{1cm}
\item[Case (a).]
$P$ is any proper parabolic subgroup when $\operatorname{rank}_\mathbb{R}\mathfrak{g}=1$. 
\item[Case (b).]
$P$ is any proper parabolic subgroup when $\mathfrak{g}=\mathfrak{su}(p,q)$ $(p+q\leq5)$, $\mathfrak{so}(p,q)\ (p+q\leq6)$, $\mathfrak{sp}(p,q)\ (p+q\leq4)$, $\mathfrak{e}_{6(2)},\mathfrak{e}_{6(-14)},\mathfrak{e}_{6(-26)},\mathfrak{f}_{4(4)}$, $\mathfrak{f}_{4,\mathbb{C}},\mathfrak{g}_{2(2)}$, or $\mathfrak{g}_{2,\mathbb{C}}$.
\item[Case (c).]
$\mathfrak{g}$ is complex simple or split. 
The semisimple part $[\mathfrak{l},\mathfrak{l}]$ of $\mathfrak{l}$ is not in the list of Table \ref{tab:tempsplit}. 
\item[Case (d).]
$\mathfrak{g}$ is neither complex nor split. The semisimple part $[\mathfrak{l},\mathfrak{l}]$ or its non-compact semisimple part $\mathfrak{l}_{ns}$ is not in the list of Table \ref{tab:tempnonsplit}. 
\end{enumerate}
\end{itemize}
\end{theorem}

\begin{table}[!h]
\caption{$\mathfrak{g}$ is complex or split}
\label{tab:tempsplit}
\[
{\arraycolsep=5mm
\begin{array}{|c|cc|}
\hline
\mathfrak{g}&[\mathfrak{l},\mathfrak{l}]&\\\hline
\mathfrak{a}_n&\mathfrak{a}_{n_1}\oplus\cdots\oplus\mathfrak{a}_{n_k}&2\max_{1\leq j\leq k}n_j\geq n+1\\
\mathfrak{b}_n&\mathfrak{a}_{n_1}\oplus\cdots\oplus\mathfrak{a}_{n_k}\oplus\mathfrak{b}_m&2m\geq n+1\\
\mathfrak{c}_n&\mathfrak{a}_{n_1}\oplus\cdots\oplus\mathfrak{a}_{n_k}\oplus\mathfrak{c}_m&2m\geq n+1\\
\mathfrak{d}_n&\mathfrak{a}_{n_1}\oplus\cdots\oplus\mathfrak{a}_{n_k}\oplus\mathfrak{d}_m&2m\geq n+2\\
\mathfrak{d}_n&\mathfrak{a}_{n-1}&n\geq3\\
\mathfrak{e}_6&\mathfrak{d}_5&\\
\mathfrak{e}_7&\mathfrak{d}_6\text{ or }\mathfrak{e}_6&\\
\mathfrak{e}_8&\mathfrak{e}_7&\\\hline
\end{array}}
\]
\end{table}

\begingroup
\renewcommand{\arraystretch}{1.2}
\begin{table}[!h]
\caption{$\mathfrak{g}$ is neither complex nor split}
\label{tab:tempnonsplit}
\[
{\arraycolsep=5mm
\begin{array}{|c|cc|}
\hline
\mathfrak{g}&\mathfrak{l}_{ns}&\\\hline
\mathfrak{su}(p,q)&\mathfrak{su}(p-k,q-k)&1\leq k\leq \min\left(p-1,q-1,\frac{p+q-2}{4}\right)\\
\mathfrak{so}(p,q)&\mathfrak{so}(p-k,q-k)&1\leq k\leq \min\left(p-1,q-1,\frac{p+q-3}{4}\right)\\
\mathfrak{sp}(p,q)&\mathfrak{sp}(p-k,q-k)&1\leq k\leq \min\left(p-1,q-1,\frac{p+q-1}{4}\right)\\\hline
\mathfrak{g}&[\mathfrak{l},\mathfrak{l}]&\\\hline
\mathfrak{su}^*(2n)&\bigoplus_{j=1}^k\mathfrak{su}^*(2m_j)&2\max_{1\leq j\leq k} m_j\geq n+2\\
\mathfrak{so}^*(4n)&\mathfrak{su}^*(2n)&n\geq2\\
\mathfrak{so}^*(2n)&\mathfrak{so}^*(2m)\oplus\bigoplus_{j=1}^k\mathfrak{su}^*(2m_j)&m\geq n+2\\
\mathfrak{e}_{7(-5)}&\mathfrak{so}^*(12)&\\
\mathfrak{e}_{7(-25)}&\mathfrak{so}(2,10)\text{ or }\mathfrak{e}_{6(-26)}&\\
\mathfrak{e}_{8(-24)}&\mathfrak{e}_{7(-25)}&\\\hline
\end{array}}
\]
\end{table}
\endgroup

\newpage

\begin{proof}
Since the diagonal map $G\to G\times G$ induces an open dense embedding $\iota\colon G/L\hookrightarrow G/P\times G/Q$, the tensor product representation $\operatorname{Ind}^G_P(\sigma)\otimes\operatorname{Ind}^G_Q(\tau)$ is unitarily equivalent to $\operatorname{Ind}^G_L(\sigma\otimes\tau)$ via the pullback $\iota^*$. Then the equivalence (i) $\iff$ (ii) follows from the Herz majoration principle (Lemma \ref{lem:Herz}) as in Theorem \ref{thm:tensor}.

To see the equivalence (ii) $\iff$ (iii), we may and do assume that $G$ is an algebraic Lie group without loss of generality by \cite[Cor.\ 3.3 and Rem.\ 3.4]{BK-II}.
We shall write $G_\mathbb{C}$ for a complex Lie group which contains $G$ as a real form.

The tensor product representation $\operatorname{Ind}^G_P(\mathbf{1})\otimes\operatorname{Ind}^G_Q(\mathbf{1})$ is unitarily equivalent to $L^2(G/L)$ via the pullback $\iota^*$. 
So our main task is to give a classification of the Levi subgroups $L$ of $G$ such that the regular representation on $L^2(G/L)$ is non-tempered.

We divide the proof into the following cases.
\begin{itemize}
\setlength{\itemindent}{1.2cm}
\item[Case I.] $G$ is complex or split. 
\setlength{\itemindent}{1.3cm}
\item[Case II.] $G$ is neither complex nor split.
\begin{itemize}
\setlength{\itemindent}{1.7cm}
\item[Case II-a.] $\mathfrak{g}\neq\mathfrak{sl}(2n-1,\mathbb{H}),\mathfrak{e}_{6(-26)}$, or $\mathfrak{e}_{6(-14)}$.
\item[Case II-b.] $\mathfrak{g}=\mathfrak{sl}(2n-1,\mathbb{H}),\mathfrak{e}_{6(-26)}$, or $\mathfrak{e}_{6(-14)}$.
\end{itemize}
\end{itemize}

{\bf Case I.} $G$ is complex or split. In this case one can read the list of the pairs $(G,L)$ such that $L^2(G/L)$ is non-tempered from the classification results of tempered reductive homogeneous spaces in \cite[Thms.\ 3.1 and 4.1]{BK-III} and from a description of Levi subgroups $L$ of complex Lie algebras by the Dynkin diagram. 
We illustrate the argument by taking $\mathfrak{g}$ to be $\mathfrak{a}_n^\mathbb{C},\mathfrak{e}_8^\mathbb{C}$ or their split real forms as examples.
For instance, let $\mathfrak{g}=\mathfrak{sl}_{n+1}(\mathbb{C})$ and $\mathfrak{l}$ be any Levi subalgebra. Then the semisimple part $[\mathfrak{l}, \mathfrak{l}]$ of $\mathfrak{l}$ is of the form $\mathfrak{sl}_{m_1}(\mathbb{C})\oplus\cdots\oplus\mathfrak{sl}_{m_k}(\mathbb{C})$ for some $m_1, \ldots, m_k(\geq2)$ with $m_1+\cdots+m_k\leq n+1$. 
By \cite[Cor.\ 3.3]{BK-II}, $L^2(G/L)$ is non-tempered if and only if $L^2(G/[L, L])$ is non-tempered. By \cite[Thm.\ 3.1]{BK-III}, this happens if and only if $[\mathfrak{l}, \mathfrak{l}]$ contains $\mathfrak{sl}_p(\mathbb{C})$ as an ideal for some $p$ with $p\geq (n+1-p)+2$ or coincides with $\mathfrak{sp}_p(\mathbb{C})$ ($2p=n+1$). The former happens when $2\max m_j\geq n+3$ and the latter never happens in our setting. Putting $n_j =m_j-1$, we conclude that $\operatorname{Ind}^G_P(\mathbf{1})\otimes\operatorname{Ind}^G_Q(\mathbf{1})$ is non-tempered if and only if $2\max n_j \geq n+1$. 
The same conclusion holds if $\mathfrak{g}$ is the split real form $\mathfrak{sl}_{n+1}(\mathbb{R})$ of $\mathfrak{sl}_{n+1}(\mathbb{C})$ by \cite[Prop.\ 5.2]{BK-I}. 
This shows the first row in the Table 1.
Of course, the conclusion matches Theorem \ref{thm:tensor}
because $d(P)+d(Q)=2\max m_j=2\max(n_j+1)$
(We note that $n_j$ and $n$ in Theorem \ref{thm:tensor} are $n_j+1$ and $n+1$ here. )

If $\mathfrak{g}$ is a complex simple Lie algebra $\mathfrak{e}_8^\mathbb{C}$, then the classification in \cite[Thm.\ 4.1]{BK-III} tells us that $L^2(G/H)$ is non-tempered if and only if $\mathfrak{e}_7^\mathbb{C}\subset \mathfrak{h}\subset\mathfrak{e}_7^\mathbb{C}\oplus\mathfrak{sl}_2(\mathbb{C})$ when $H$ is a (proper) complex reductive subgroup of $G$. On the other hand, the Dynkin diagram of type $E_8$ shows that a Levi subalgebra $\mathfrak{l}$ of $\mathfrak{e}_8^\mathbb{C}$ containing $\mathfrak{e}_7^\mathbb{C}$ is either $\mathfrak{e}_7^\mathbb{C}\oplus\mathbb{C}$ or $\mathfrak{e}_8^\mathbb{C}$ itself. This gives the last row in Table 1. The same conclusion holds for the split real form by \cite[Prop.\ 5.2]{BK-I}.
Table 1 for other (complex or split) simple Lie algebras is obtained similarly by using the Dynkin diagram and \cite[Thms.\ 3.1 and 4.1]{BK-III}.

{\bf Case II.} $G$ is neither complex nor split. 
We recall from \cite[Prop.\ 3.1]{BK-II} that $L^2(G/L)$ is tempered if and only if $L^2(G/L_{ns})$ is tempered. 
Thus the condition (ii) is equivalent to that $L^2(G/L_{ns})$ is tempered. 
We note that the non-compact semisimple factor $L_{ns}$ may be much smaller than $L$ in Case II.  
Accordingly, it may well happen that $L^2(G/L)$ is tempered but $L^2(G_\mathbb{C}/L_\mathbb{C})$ is not tempered. 
This means that the tensor product representation $\operatorname{Ind}^G_P(\mathbf{1})\otimes\operatorname{Ind}^G_Q(\mathbf{1})$ in Case II is more likely to be tempered than $\operatorname{Ind}_{P_\mathbb{C}}^{G_\mathbb{C}}(\mathbf{1})\otimes\operatorname{Ind}_{Q_\mathbb{C}}^{G_\mathbb{C}}(\mathbf{1})$ which was treated in Case I. 

For example, if $\operatorname{rank}_\mathbb{R}G=1$, then any (proper) parabolic subgroup is a minimal parabolic subgroup, hence $L_{ns}=\{e\}$ and thus $L^2(G/L)$ is tempered. 
For the computation of $\mathfrak{l}_{ns}$ in the general case, we can use the Satake diagram, which we recall now, see \cite[Chap.\ 10]{Hl} for example. 

Let $\mathfrak{g}=\mathfrak{k}+\mathfrak{p}$ be a Cartan decomposition, $\mathfrak{a}$ a maximal abelian subspace in $\mathfrak{p}$, and extend $\mathfrak{a}$ to a Cartan subalgebra $\mathfrak{j}$ of $\mathfrak{g}$. We take compatible positive systems $\Delta^+(\mathfrak{g}_\mathbb{C},\mathfrak{j}_\mathbb{C})$ and $\Sigma^+(\mathfrak{g},\mathfrak{a})$ such that
$\alpha|_\mathfrak{a}\in\Sigma^+(\mathfrak{g},\mathfrak{a})\cup\{0\}$, whenever $\alpha\in\Sigma^+(\mathfrak{g}_\mathbb{C},\mathfrak{j}_\mathbb{C})$. Then one has a surjective map $r\colon\Psi\to\Phi\cup\{0\}$ where $\Psi$ and $\Phi$ are the sets of simple roots of $\Delta^+(\mathfrak{g}_\mathbb{C},\mathfrak{j}_\mathbb{C})$ and $\Sigma^+(\mathfrak{g},\mathfrak{a})$, respectively. The Satake diagram is an enriched Dynkin diagram for $\Psi$ by coloring $r^{-1}(\{0\})$ black and by connecting two white nodes $\alpha\neq\beta$ by arrows if $r(\alpha)=r(\beta)\neq0$. Any Levi subalgebra $\mathfrak{l}$ of a real semisimple Lie algebra $\mathfrak{g}$ is conjugate to $\bigoplus\mathfrak{g}(\mathfrak{a};\lambda)$ by an inner automorphism of $\mathfrak{g}$, where the sum is taken over all $\lambda$ in the $\mathbb{Z}$-span of $S$ for some subset $S$ of $\Phi$ . Then the Dynkin diagram for $\Delta(\mathfrak{l}_\mathbb{C},\mathfrak{j}_\mathbb{C})$ of the complexified Lie algebra $\mathfrak{l}_\mathbb{C}$ is given by $r^{-1}(S\cup\{0\})$. Let $V$ be the union of the connected components in $r^{-1}(S\cup\{0\})$ that consist of black nodes in the Satake diagram. Then the Dynkin diagram for $(\mathfrak{l}_{ns})_\mathbb{C}$ is given by $r^{-1}(S\cup\{0\})\setminus V$. With this in mind, we apply the classification theory in \cite{BK-III} as follows. 

{\bf Case II-a. }Assume $\mathfrak{g}$ is not isomorphic to $\mathfrak{sl}(2n-1, \mathbb{H}) (\simeq\mathfrak{su}^*(4n-2))$, $\mathfrak{e}_{6(-26)}$, or $\mathfrak{e}_{6(-14)}$.
Then \cite[Thm.\ 1.4]{BK-III} shows that the following conditions are equivalent:
\begin{itemize}
\item $L^2(G/L_{ns})$ is a tempered representation of $G$;
\item $L^2(G_\mathbb{C}/(L_{ns})_\mathbb{C})$ is a tempered representation of $G_\mathbb{C}$. 
\end{itemize}
In this case, we apply the classification result in \cite[Thms.\ 3.1 and 4.1]{BK-III} to the complex homogeneous space $G_\mathbb{C}/(L_{ns})_\mathbb{C}$. 
We illustrate the argument by taking $\mathfrak{g}=\mathfrak{su}(p,q)$ and real forms of $\mathfrak{e}_7^\mathbb{C}$ as examples. 
First, let us consider $\mathfrak{g}=\mathfrak{su}(p, q)\ (p\geq q)$. Then any Levi subalgebra $\mathfrak{l}$ of $\mathfrak{g}$ is of the form $\mathfrak{l}\simeq\bigoplus_{j=1}^\ell\mathfrak{gl}_{m_j}(\mathbb{C})\oplus\mathfrak{su}(p-k, q-k)$ where $m_1+\cdots+m_\ell=k\leq q$. 
Accordingly, the complexification of $\mathfrak{l}_{ns}$ is given as $(\mathfrak{l}_{ns})_\mathbb{C}\simeq\bigoplus_{j=1}^\ell(\mathfrak{sl}_{m_j}(\mathbb{C})\oplus\mathfrak{sl}_{m_j}(\mathbb{C}))\oplus\mathfrak{sl}_{p+q-2k}(\mathbb{C})$ if $q>k$ and $(\mathfrak{l}_{ns})_\mathbb{C}\simeq\bigoplus_{j=1}^\ell(\mathfrak{sl}_{m_j}(\mathbb{C})\oplus\mathfrak{sl}_{m_j}(\mathbb{C}))$ if $q=k$. Thus \cite[Thm.\ 3.1]{BK-III} implies that $L^2(G_\mathbb{C}/(L_{ns})_\mathbb{C})$ is non-tempered if and only if $2(p+q-2k)\geq p+q+2$ and $(p\geq)q>k$. This shows the first row in Table 2. 

Next let us treat real forms of $\mathfrak{e}_7^\mathbb{C}$. 
By the classification \cite[Thm.\ 4.1]{BK-III}, for any real form $\mathfrak{g}$ of $\mathfrak{e}_7^\mathbb{C}$, $L^2(G_\mathbb{C}/(L_{ns})_\mathbb{C})$ is non-tempered if and only if $(\mathfrak{l}_{ns})_\mathbb{C}$ contains $\mathfrak{d}_6^\mathbb{C}$ or $\mathfrak{e}_6^\mathbb{C}$.
There are four real forms of $\mathfrak{e}_7^\mathbb{C}$, namely, a compact real form, $\mathfrak{e}_{7(7)}(=\text{EV})$, $\mathfrak{e}_{7(-5)}(=\text{EVI})$,  and $\mathfrak{e}_{7(-25)}(=\text{EVII})$. The second one is split, and was treated in Case I. For the remaining two real forms, the Satake diagrams are given as below. 
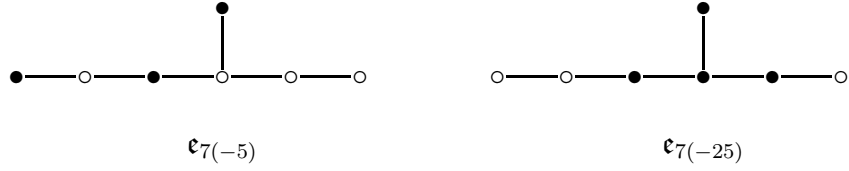
\[
\xymatrix@!=2pt@M=0.5pt{
&&&\bullet\ar@{-}[d]&&&&&&&\bullet\ar@{-}[d]&&\\
\bullet\ar@{-}[r]&\circ\ar@{-}[r]&\bullet\ar@{-}[r]&\circ\ar@{-}[r]&\circ\ar@{-}[r]&\circ&&\circ\ar@{-}[r]&\circ\ar@{-}[r]&\bullet\ar@{-}[r]&\bullet\ar@{-}[r]&\bullet\ar@{-}[r]&\circ\\
&&&\mathfrak{e}_{7(-5)}&&&&&&&\mathfrak{e}_{7(-25)}&&
}
\]
Then the non-compact semisimple part $\mathfrak{l}_{ns}$ of a (real) Levi subalgebra $\mathfrak{l}$ having the property $(\mathfrak{l}_{ns})_\mathbb{C}\supset\mathfrak{d}_6^\mathbb{C}\simeq\mathfrak{so}_{12}(\mathbb{C})$ or $\mathfrak{e}_6^\mathbb{C}$ is listed as follows.
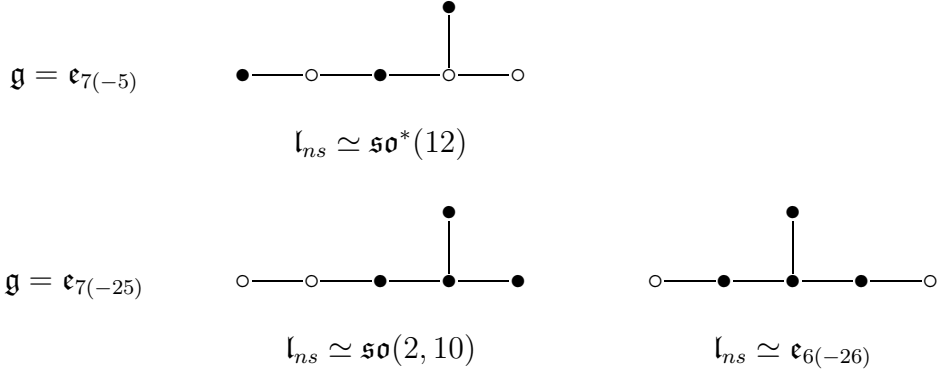
\[
\xymatrix@!=2pt@M=0.5pt{
&&&&&\bullet\ar@{-}[d]&&&&&&&&\\
\mathfrak{g}=\mathfrak{e}_{7(-5)}\quad\quad&&\bullet\ar@{-}[r]&\circ\ar@{-}[r]&\bullet\ar@{-}[r]&\circ\ar@{-}[r]&\circ\ar@{-}&&&&&&&\\
&&&&\mathfrak{l}_{ns}\simeq\mathfrak{so}^*(12)&&&&&&&\\
&&&&&\bullet\ar@{-}[d]&&&&&\bullet\ar@{-}[d]&&\\
\mathfrak{g}=\mathfrak{e}_{7(-25)}\quad\quad&&\circ\ar@{-}[r]&\circ\ar@{-}[r]&\bullet\ar@{-}[r]&\bullet\ar@{-}[r]&\bullet\ar@{-}&&\circ\ar@{-}[r]&\bullet\ar@{-}[r]&\bullet\ar@{-}[r]&\bullet\ar@{-}[r]&\circ\ar@{-}\\
&&&&\mathfrak{l}_{ns}\simeq\mathfrak{so}(2,10)&&&&&&\mathfrak{l}_{ns}\simeq\mathfrak{e}_{6(-26)}&
}
\]
This shows the last two and three rows in Table 2.

{\bf Case II-b.} $\mathfrak{g}=\mathfrak{su}^*(4m-2)$, $\mathfrak{e}_{6(-26)}$, or $\mathfrak{e}_{6(-14)}$. In this case, it may happen that $L^2(G/H)$ is tempered but $L^2(G_\mathbb{C}/H_\mathbb{C})$ is not tempered for some reductive subgroup $H$ even when $H=H_{ns}$, see \cite[Thm.\ 1.4 (ii)-(iv)]{BK-III} for the list of such $H$. 
We need to take this exceptional case into account if such $H$ arises as the non-compact semisimple part $L_{ns}$ of a Levi subgroup $L$ of $G$. 
For example, suppose $\mathfrak{g}=\mathfrak{sl}(n,\mathbb{H})\ (\simeq\mathfrak{su}^*(2n))$. Then any Levi subalgebra $\mathfrak{l}$ of $\mathfrak{g}$ is of the form
\[
\mathfrak{gl}(m_1,\mathbb{H})\oplus\cdots\oplus\mathfrak{gl}(m_k,\mathbb{H})\simeq\bigoplus_{j=1}^k\mathfrak{su}^*(2m_j)\oplus\mathbb{R}^k, 
\]
where $m_1+\cdots+m_k=n$.
We may and do assume that $m_j>1$ for $1\leq j\leq \ell$ and $m_j=1$ for $\ell+1\leq j\leq k$. Then $\mathfrak{l}_{ns}\simeq\bigoplus_{j=1}^\ell\mathfrak{su}^*(2m_j)$ because $\mathfrak{su}^*(2)\simeq\mathfrak{su}(2)$.
By \cite[Thm.\ 1.4 (iii)]{BK-III}, the exceptional case occurs when $\max_{1\leq j\leq \ell}m_j=\frac{1}{2}(n+1)$, namely, $L^2(G/L_{ns})$ is non-tempered if and only if $\max_{1\leq j\leq \ell}m_j\neq \frac{1}{2}(n+1)$ and $L^2(G_\mathbb{C}/(L_{ns})_\mathbb{C})$ is non-tempered. The latter condition amounts to $\max_{1\leq j\leq \ell}m_j\geq\frac{1}{2}(n+1)$ by \cite[Thm.\ 3.1]{BK-III} (see also \cite[Ex.\ 8.8]{BK-III}) because $(\mathfrak{l}_{ns})_\mathbb{C}\simeq\bigoplus_{j=1}^\ell\mathfrak{sl}_{2m_j}(\mathbb{C})$. Hence $L^2(G/L)$ is non-tempered if and only if $2\max_{1\leq j\leq \ell}m_j>n+1$, or equivalently,  $2\max_{1\leq j\leq k}m_j>n+1$, as listed in Table 2. Other cases are similar and easier. 

This completes the proof of Theorem \ref{thm:append}.
\end{proof}

\section*{Acknowledgement}
The authors are grateful to the IHES and to The University of Tokyo for its support through the GCOE program.
The third author was partially supported by Grant-in-Aid for Scientific Research (A) 18H03669,  JSPS. 
We thank the anonymous referee for reading carefully the manuscript

{\small
\noindent
Y. \textsc{Benoist} :
CNRS-Universit\'e Paris-Saclay, Orsay, France, \newline
e-mail : \texttt{yves.benoist@math.u-psud.fr}

\medskip
\noindent
Y. \textsc{Inoue} : 
Graduate School of Mathematical Sciences, 
The University of Tokyo, Komaba, Japan, 
e-mail : \texttt{inoue-78u@gl.pen-kanagawa.ed.jp}

\medskip
\noindent
T. \textsc{Kobayashi} :
Graduate School of Mathematical Sciences, 
The University of Tokyo, Komaba,  Japan, 
e-mail : \texttt{toshi@ms.u-tokyo.ac.jp}

\end{document}